\documentclass[12pt, leqno]{amsart}
\usepackage{latexsym}
\usepackage{amssymb,amsfonts,amsmath,mathrsfs}
\usepackage{hyperref}
\newtheorem{theorem}{Theorem}[section]

\newtheorem{proposition}[theorem]{Proposition}

\newtheorem{lemma}[theorem]{Lemma}

\newtheorem*{theorem1.2}{Theorem 1.2}
\newtheorem*{theorem1.3}{Theorem 1.3}
\newcommand{\tmop}[1]{\ensuremath{\operatorname{#1}}}
\numberwithin{equation}{section}

\begin{document}

\title[On the variance of sums of divisor functions in short intervals]{On the variance of sums of divisor functions \\ in short intervals}

\author[S. Lester]{Stephen Lester}
\address{School of Mathematical Sciences, Tel Aviv University, Tel Aviv 69978, Israel}
\email{slester@post.tau.ac.il}

\thanks{The research leading to these results has received funding
from the European Research Council under the European Union's
Seventh Framework Programme (FP7/2007-2013) / ERC grant agreement
n$^{\text{o}}$ 320755.}

\date{}

\begin{abstract}
Given a positive integer $n$ the $k$-fold divisor function $d_k(n)$ equals the number of ordered
$k$-tuples of positive integers whose product equals $n$. In this article we study the variance
of sums of $d_k(n)$ in short intervals 
and establish asymptotic formulas for the variance of sums of $d_k(n)$ 
in short intervals of certain lengths for $k=3$
and for $k \ge 4$ under the assumption of the Lindel\"of hypothesis.
\end{abstract}


\maketitle

\section{Introduction and main results}

Let
$k \ge 2$ be an integer and $d_k(n)$ denote the number of ordered $k$-tuples of positive integers whose
product is $n$. Also, write
\[
\Delta_k(x)
=\sum_{n \le x} d_k(n) - \tmop{Res}_{s=1} \Big( \zeta^k(s) \frac{x^s}{s}\Big),
\]
where $\zeta(s)$ is the Riemann zeta-function and the residue on the right-hand side equals $x R(\log x)$
where $R(x)$ is a polynomial of degree $k-1$.

Asymptotic formulas for the mean square of $\Delta_k(x)$, which is the variance
of sums of $d_k(n)$ with $1 \le n \le x$,
have been given by Cram\'er ~\cite{Cramer}, with $k=2$,
and
Tong \cite{tong} for $k =3$ as well as for $k \ge 4$ under
assumption of the Lindel\"of
hypothesis. In this article we study the variance of sums of $d_k(n)$ in short intervals. Short intervals with $x< n < x+h$ and $h=o(x)$ capture the erratic nature of $d_k(n)$ better than long intervals do
and the variance of sums of $d_k(n)$ over short intervals gives stronger information about its behavior.
Additionally, it has long been understood that there is a connection between
the variance of sums of $d_k(n)$ and the $2k$th moment of the Riemann zeta-function. This connection becomes more pronounced when looking at short intervals.

We first note that from the previously stated estimates for the variance
of sums of $d_k(n)$ with $1 \le n \le x$ it follows that
for $k=2,3$, and, for $k \ge 4$ assuming the Lindel\"of hypothesis that
\begin{equation} \label{variance}
\frac1X \int_{X}^{2X} \Big(\Delta_k(x+H)-\Delta_k(x)\Big)^2
\, dx \ll X^{1-\frac1k},
\end{equation}
for $2 \le H \le X$.
When $H$ is small this bound is not very good
and one expects that when $H =o (X^{1-\frac1k})$ this can be improved. 
Using a method of Selberg \cite{Selberg}, 
Milinovich and Turnage-Butterbaugh \cite{MT} have
given an elegant argument, which provides a better bound than \eqref{variance} for $H=o(X^{1-\frac1k})$. Assuming
the Riemann hypothesis and
applying Harper's \cite{Harper} sharp refinement of Soundararajan's \cite{Sound}
bound for the $2k$th moment of the Riemann zeta-function, their argument shows that
\begin{equation} \label{var}
\frac1X \int_{X}^{2X} \Big(\Delta_k\Big(x+\frac{x \cdot X^{-\frac1k}}{L}\Big)-\Delta_k(x)\Big)^2,
\, dx \ll \frac{X^{1-\frac1k}}{ L}(\log  L)^{k^2}
\end{equation}
where $2 \le L \ll X^{1-\frac1k-\varepsilon}$ for some $\varepsilon>0$.
However, this upper bound is most likely not sharp and the true order of magnitude is probably of size $(X^{1-\frac{1}{k}}/L) \cdot (\log L)^{k^2-1}$. 

More precise estimates than \eqref{var} in the case that $k=2$ are given by
Jutila ~\cite{Jutila2} and Ivi\'c ~\cite{Ivic} (see also \cite{coppola} and \cite{LY}). In particular, Ivi\'c \cite{Ivic}
derives
an explicit asymptotic
formula
for the variance of sums of $d_2(n)=d(n)$
in short intervals with  $x < n \le x+h$ and $x^{\varepsilon} \ll h \ll x^{\frac12-\varepsilon}$. 
For  $X^{\varepsilon} \ll L \ll X^{\frac12-\varepsilon}$,
for some $\varepsilon>0$, Ivi\'c proves that
\begin{equation} \label{form ivic}
\frac1X \int_X^{2X}
\Big(\Delta_2\Big(x+\frac{X^{1/2}}{L}\Big)-\Delta_2(x) \Big)^2 \, dx
=\frac{8}{\pi^2} \frac{X^{1/2}}{L} (\log L )^3+O\Big(\frac{X^{1/2}}{L} (\log L)^2 \Big).
\end{equation}
A main tool in Jutila's method for estimating the variance of $d(n)$ in short intervals
is the Voronoi summation formula. This formula expresses $\Delta_2(x)$
in terms of a trigonometric polynomial and similar formulas can be established
for $\Delta_k(x)$ (see for instance \cite{FriedIwaniec}). However, as $k$ becomes
larger these trigonometric polynomials become more complex and even when $k=3$ this method seems to no longer 
work.

In this article we derive asymptotic formulas for the variance
of sums of $d_k(n)$ in short intervals of certain lengths for $k=3$
and
under the assumption of the Lindel\"of hypothesis  for $k\ge 4$. Our main innovation is to (essentially) combine
Jutila's approach with Selberg's method. This enables us to handle the large frequencies
in the trigonometric polynomial approximation to $\Delta_k(x)$ that are a significant obstacle in this problem.
Our formulas only hold for intervals of certain lengths and computing this variance
in even shorter intervals than those in Theorem \ref{main thm} seems difficult and would be very interesting.

Let
\begin{equation} \label{constants}
a_k= \prod_p
\Big( (1-p^{-1})^{k^2} \sum_{j=0}^{\infty}\Big(\frac{\Gamma(k+j)}{\Gamma(k)j!}\Big)^2\frac{1}{p^{j}} \Big)
\quad  \mbox{and} \quad C_k=\frac{2^{2-\frac1k}-1}{2-\frac1k}  \cdot \frac{k^{k^2-1}}{\Gamma(k^2)}\cdot a_k.
\end{equation}
Our first main result gives an estimate for the
variance of sums of $d_3(n)$ in short intervals of certain lengths. 

\begin{theorem} \label{3thm} Suppose that
$2 \le L \ll X^{\frac{1}{12}-\varepsilon}$ for some fixed $\varepsilon>0$.
Then
\[
\frac1X \int_X^{2X}
\Big(\Delta_3\Big(x+\frac{x^{2/3}}{L}\Big)-\Delta_3(x) \Big)^2
\, dx
= C_3   \cdot \frac{X^{2/3}}{L} (\log L)^8+O\Big(\frac{X^{2/3}}{L} (\log L)^7 \Big).
\]
\end{theorem}

We also examine the variance of sums of $d_k(n)$ in short intervals
for $k \ge 3$ under the assumption of the Lindel\"of hypothesis.

\begin{theorem} \label{main thm}
Assume the Lindel\"of hypothesis. Suppose
$2 \le L \ll X^{\frac{1}{k(k-1)}-\varepsilon}$ for some fixed $\varepsilon>0$.
Then for each integer $k \ge 3$ we have 
\begin{equation} \notag
\begin{split}
\frac1X \int_X^{2X}
\Big(\Delta_k\Big(x+\frac{x^{1-\frac1k}}{L}\Big)-\Delta_k(x) \Big)^2
\, dx
=& C_k \cdot \frac{X^{1-\frac1k}}{L}(\log L)^{k^2-1}\\
 &+O\Big(\frac{X^{1-\frac1k}}{L}(\log L)^{k^2-2}\Big).
\end{split}
\end{equation}
\end{theorem}

\noindent
\textbf{Remark.}
There is a slight difference between $C_2$ and the leading
coefficient in \eqref{form ivic}.
This arises because the lengths of our intervals depend on the variable $x$.  To clarify this discrepancy note that
\[
C_k= \bigg(\int_1^2 x^{1-\frac1k} \, dx \bigg)  \cdot \frac{k^{k^2-1}}{\Gamma(k^2)}\cdot a_k.
\]

Ivi\'c \cite{Ivic} gives
a more precise formula for the left-hand side of \eqref{form ivic} that includes lower order terms and an error term with a power savings in $X$.  We can also prove more
precise formulas than those stated in Theorems \ref{3thm} and \ref{main thm}.  In particular, assuming the Lindel\"of hypothesis, we 
can show for each $k \ge 3$ and $L=X^{\delta}$ with $\varepsilon < \delta < \frac{1}{k(k-1)}-\varepsilon$, 
for some $\varepsilon>0$, that
\[
\frac1X \int_X^{2X}
\Big(\Delta_k\Big(x+\frac{x^{1-\frac1k}}{L}\Big)-\Delta_k(x) \Big)^2
\, dx
= \sum_{j=0}^{k^2-1} c_j \cdot \frac{X^{1-\frac1k}}{L}(\log L)^{j} +O\Big(\frac{X^{\vartheta(\delta,k)}}{L}\Big)
\] 
where $\vartheta(\delta,k)\le 1-\frac1k-\eta$
for some fixed $\eta=\eta(\varepsilon)>0$. We have
not computed the coefficients $c_j$ for
$0\le j \le k^2-2$.

In concurrent work, Keating, Rodgers, Roditty-Gershon, and Rudnick (see \cite{KRRR})
have established strong results on an analog of this problem
in the setting of function fields over a finite field $\mathbb F_q$ in the case $q \rightarrow \infty$
and degree of the polynomials, $n$, is fixed. 
In this setting they succeed in unconditionally computing the variance of sums of divisor functions in very short intervals. From their results we expect that the order of the left-hand side of \eqref{var}
to be of size $(X^{1-\frac1k}/{L}) \cdot (\log L)^{k^2-1}$ for $X^{\varepsilon} \ll L \ll X^{1-\frac1k-\varepsilon}$. 
Additionally, their analysis suggests that leading order constant should have a very elaborate and interesting behavior.
For instance, a transition appears there when the lengths of the intervals become smaller than those
considered in Theorem \ref{main thm}. This is consistent with our analysis, since when the intervals become shorter than those considered in Theorem \ref{main thm} our method fails for several reasons. Not only does the polynomial approximation 
to $\Delta_k(x+h)-\Delta_k(x)$
become too long to handle, but
it also seems to no longer effectively approximate $\Delta_k(x+h)-\Delta_k(x)$ in mean square. These breaking points coincide precisely at this transition.

\section{Main propositions}
Our first main step approximates
$\Delta_k(x)$ on average by short trigonometric polynomials. This may be compared to what can be proved for pointwise approximations (see \cite{FriedIwaniec}).
\begin{proposition} \label{3prop}
Let $0<\theta\le \frac{1}{2}$ and
\[
P_3(x;\theta)=\frac{x^{\frac{1}{3}}}{\pi \sqrt{3}}
\sum_{n \le  X^{\theta}}
\frac{d_3(n)}{n^{\frac23}}\cos\big(6\pi \sqrt[3]{nx}\big).
\]
Then for any $\varepsilon>0$ we have
\begin{equation} \label{poly approx}
\frac{1}{X} \int_{X}^{2X}
\Big(\Delta_3(x)-P_3(x;\theta) \Big)^2 \, dx
\ll X^{\frac23-\frac{\theta}{6}+\varepsilon}.
\end{equation}
\end{proposition}

Assuming the Lindel\"of hypothesis we are able to prove a stronger result for the ternary divisor function
as well gives analogous result for $d_k(n)$ for each $k \ge 4$.

\begin{proposition} \label{main prop}
Assume the Lindel\"of hypothesis.  Let
$k \ge 3$ be an integer, $0<\theta \le \frac{1}{k-1}$, and
\[
P_k(x;\theta)=\frac{x^{\frac{1}{2}-\frac{1}{2k}}}{\pi \sqrt{k}}
\sum_{n \le  X^{\theta}}
\frac{d_k(n)}{n^{\frac12+\frac{1}{2k}}}\cos\Big(2\pi k \sqrt[k]{nx}+\frac{k-3}{4}\pi\Big).
\]
Then for any $\varepsilon>0$ we have
\begin{equation} \label{varbdbest}
\frac{1}{X} \int_{X}^{2X}
\Big(\Delta_k(x)-P_k(x;\theta) \Big)^2 \, dx
\ll X^{1-\frac{1}{k} \cdot (1+\theta)+\varepsilon}.
\end{equation}
\end{proposition}

For $k=2$ and $0<\theta \le 1$ the inequality \eqref{varbdbest} is known to hold unconditionally (see equation (12.4.4) of Titchmarsh \cite{Titchmarsh}).
The strength of the upper bound is significant and
better bounds in \eqref{poly approx}
correspond to 
being able to compute the variance of sums of $d_k(n)$ in shorter intervals. In Proposition \ref{main prop}
we obtain a better estimate in the case $k=3$ 
than \eqref{poly approx}. This allows us to compute the variance of sums of $d_3(n)$
in even shorter intervals, assuming the Lindel\"of hypothesis. 
Heath-Brown ~\cite{Heath-Brown}
has also obtained an estimate for the left-hand
side of  \eqref{poly approx} by estimating
the mean values of $\Delta_3(x)P_3(x;\theta)$ and $P_3(x;\theta)^2$ 
and then applying Tong's formula for the mean square
of $\Delta_3(x)$. Our upper bound strengthens the estimate
given by Heath-Brown. 
Additionally, our proof of \eqref{poly approx} is significantly different. Particularly,
it does not use Tong's results. In fact, our argument gives a new proof of Tong's formulas.

As another application of the above propositions we will establish asymptotic formulas
for the variance of sums of $d_k(n)$ in intervals with $x<n\le x+h$ 
with $x^{1-\frac1k+\varepsilon} \ll h \ll x^{1-\varepsilon}$. In this regime $\Delta_k(x+h)$ and
$\Delta_k(x)$ interact as if they are uncorrelated.
\begin{theorem} \label{notsoshort}
Suppose that $X^{1-\frac1k+\varepsilon} \ll H \ll X^{1-\varepsilon}$ for some $\varepsilon >0$. Then 
for $k =2,3$ and for $k \ge 4$ under the assumption of the Lindel\"of hypothesis
we have as $X \rightarrow \infty$
\[
\frac1X \int_{X}^{2X} \Big(\Delta_k(x+H)-\Delta_k(x)\Big)^2 \, dx \sim B_k \cdot X^{1-\frac1k}
\]
where
\[
B_k= \frac{2^{2-\frac1k}-1}{2-\frac1k} \cdot \frac{1}{\pi^2 k}\sum_{n = 1}^{\infty} \frac{d_k(n)^2}{n^{1+\frac1k}}.
\]
\end{theorem}
The leading order constant here is essentially twice the one that appears in Tong's formula \cite{tong} for the mean square of $\Delta_k(x)$. 
As we will see, this reflects
that the covariance between $\Delta_k(x+h)$ and $\Delta_k(x)$ tends to zero as $X \rightarrow \infty$ in this regime.

Before proving Propositions \ref{poly approx} and \ref{main prop} we
first require several preliminary lemmas. The first of these lemmas cites a stationary phase estimate. Here and throughout  $\chi(s)=\pi^{s-\frac12} \Gamma(\frac{1-s}{2})/\Gamma(\frac{s}{2})$ is the functional equation factor for $\zeta(s)$, that is $\zeta(s)=\chi(s)\zeta(1-s)$.
\begin{lemma}\label{phase}
Suppose $k \ge 2$. For sufficiently large $Y<x$ we have
\begin{equation} \notag
\begin{split}
\frac{1}{2\pi i} \int_{-\epsilon-iY}^{-\epsilon+iY}
  \chi^k(s)\zeta^k(1-s) x^{s}\frac{ds}{s}
=&\frac{x^{\frac{1}{2}-\frac{1}{2k}}}{\pi \sqrt{k}}\sum_{n \le N}
\frac{d_k(n)}{n^{\frac{1}{2}+\frac{1}{2k}}}\cos\Big(2\pi k \sqrt[k]{nx}+\frac{k-3}{4}\pi\Big)\\
&+O\Big(Y^{\frac{k}{2}-1} x^{\epsilon}+x^{1+\epsilon}  Y^{-\frac12-\frac{k}{2}}\Big)
\end{split}
\end{equation}
where $N= (\frac{Y}{2\pi})^k x^{-1}$.
\end{lemma}
\begin{proof}
This estimate is due to Friedlander and Iwaniec.
See pages 497-499 of ~\cite{FriedIwaniec}.
\end{proof}

Let $\lambda \ge 0$ be a number such that 
$|\zeta(\tfrac12+it)| \ll t^{\lambda+\varepsilon}$ for every $ \varepsilon >0$. It is well-known that by the Phragmen-Lindel\"of principle (or otherwise)
one has
for $\frac12 \le \sigma \le 1$ and every $\varepsilon>0$ that
\begin{equation} \label{convex}
|\zeta(\sigma+it)| \ll t^{2\lambda(1-\sigma)+\varepsilon}.
\end{equation}

\begin{lemma} \label{contour} Let $k \ge 2$ and
 $0 \le \delta \le \frac{1}{k}$. Also, let
\[
I_{k}( x;\theta, \delta)= \tmop{Re} \frac{1}{ \pi i}\int_{\frac12-\delta+iY}^{\frac12-\delta+i X}
\zeta^k(s) x^{s} \frac{ds}{s},
\]
where $Y=2 \pi  X^{\frac{1+\theta}{k}}$ and $0 \le \theta  \le \tfrac12 \cdot( k-1)$.
Then uniformly for $X \le x \le 2X$ we have
\begin{equation} \notag
\begin{split}
\Delta_k(x)
=&
\frac{x^{\frac{1}{2}-\frac{1}{2k}}}{\pi \sqrt{k}}
\sum_{nx \le  X^{1+\theta}}
\frac{d_k(n)}{n^{\frac12+\frac{1}{2k}}}\cos\Big(2\pi k \sqrt[k]{nx}+\frac{k-3}{4}\pi\Big) \\
&+I_{k}(x;\theta, \delta)+E_k(\theta, \delta, X),
\end{split}
\end{equation}
where, for any $\varepsilon>0$,
\[
E_k(\theta, \delta, X)
\ll X^{\varepsilon} \Big(X^{(\frac{1}{2}-\frac1k) \cdot(1+  \theta)}
+X^{\frac12-\frac{(1+\theta+k\theta)}{2k}}+
X^{\frac12-\delta} (X^{\frac{(1+\theta)}{k}}+X)^{k(\lambda+\delta-2\lambda \delta)-1}\Big).
\]
\end{lemma}
\noindent
\textbf{Remark.} In the bound for $E_k$ the term  $X^{(\frac12-\frac1k)(1+\theta)}$ is significant and is smaller
than $X^{\frac12-\frac{(1+\theta)}{2k}}$ only
for $\theta < \frac{1}{k-1}$. This accounts for the limitation in the range of $\theta$ in Proposition \eqref{main prop}.
\begin{proof}
Applying Perron's formula we get that
\[
\sum_{n \le x} d_k(n)=\int_{1+\epsilon-iX}^{1+\epsilon+iX}
\zeta^k(s) x^s \frac{ds}{s}+O( X^{\epsilon} ).
\]
Next, pull the contour to the line $\tmop{Re}(s)=\tfrac12-\delta$ picking up the residue at $s=1$. 
To estimate the horizontal contours, apply \eqref{convex} and use the functional equation for $\zeta(s)$ along with Stirling's formula to see that they are
$\ll X^{\varepsilon}(1+X^{\frac12-\delta}X^{k(\lambda+\delta-2\lambda \delta)-1})$.
Thus,
\begin{equation} \notag 
\begin{split}
\Delta_k(x)=&\frac{1}{2\pi i } \int_{\frac12-\delta-iY}^{\frac12-\delta+iY}
\zeta^k(s) x^s \frac{ds}{s}+I_k(x;\theta, \delta) \\
&+O\Big(X^{\varepsilon}\Big(1+X^{\frac12-\delta}X^{k(\lambda+\delta-2\lambda \delta)-1}\Big)\Big).
\end{split}
\end{equation}

Now pull the first integral on the right-hand side to the line $\tmop{Re}(s)=-\epsilon$ and
note that the residue at $s=0$ contributes $O(1)$. 
Arguing as before, the
horizontal contours are $\ll X^{\varepsilon}(Y^{\frac{k}{2}-1}+X^{\frac12-\delta}Y^{k(\lambda+\delta-2\lambda \delta)-1})$.
Finally, applying the functional equation  we see that
\[
\begin{split}
\Delta_k(x)=&\frac{1}{2\pi i}\int_{-\epsilon-iY}^{-\epsilon+iY}
\chi^k(s) \zeta^k(1-s) x^s \, \frac{ds}{s}+I_{k}(x;\theta, \delta)\\
&+O\Big(X^{\varepsilon} \Big(Y^{\frac{k}{2}-1}+X^{\frac12-\delta}
(X+Y)^{k(\lambda+\delta-2\lambda \delta)-1}\Big)\Big).
\end{split}
\]
To complete the proof apply Lemma \ref{phase}.
\end{proof}
We now show that the mean square of $I_k(x;\theta,\delta)$ can be estimated in terms of the $2k$th moment
of the Riemann zeta-function. This is essentially Plancherel's theorem and we will give a direct proof.

\begin{lemma} \label{plancherel}
Let $w$ be a smooth function that is compactly supported in the positive real numbers.
Suppose that $0\le \delta < \frac{1}{2k}$ and for every $\epsilon>0$ that
\begin{equation} \label{int asump}
\int_0^T |\zeta(\tfrac12+\delta+it)|^{2k} \, dt \ll T^{1+\epsilon}.
\end{equation}
Then we have for any $\varepsilon>0$ that
\[ 
\frac{1}{X}\int_{\mathbb R} \Big|I_{k}(x;\theta, \delta)\Big|^2 w\Big( \frac{x}{X}\Big) \, dx \ll X^{1+2\delta \theta-\frac{(1+\theta)}{k}
+\varepsilon}.
\]
\end{lemma}

\begin{proof}
Changing the order of integration and making a change of variables we get that
\begin{equation} \notag
\begin{split} 
&\frac{1}{X}\int_{\mathbb R} \Big|I_{k}(x;\theta_k, \delta)\Big|^2 w\Big( \frac{x}{X}\Big) \, dx \\
&\qquad \qquad \qquad 
\le
\frac{X^{1-2\delta}}{\pi^2}\int_Y^X \int_Y^X \frac{\zeta^k(\tfrac12-\delta+it)\zeta^k(\tfrac12-\delta-iv)}{(\tfrac12-\delta+it)(\tfrac12-\delta-iv)}
X^{i(t-v)} \mathcal I(t-v) \, dv \, dt,
\end{split}
\end{equation}
where $\mathcal I(y):=\int_{\mathbb R} u^{1-2\delta+iy} w(u) \, du$. Observe that by repeatedly
integrating by parts $\mathcal I(y) \ll_{w,A} \min(1,|y|^{-A})$. Hence,
by this, Lemma \ref{convex}, and the functional equation for $\zeta(s)$  along with Stirling's formula we get
for $U=X^{\eta}$ with $0< \eta  \le \frac1k$ fixed that
the portion of the above integral on the right-hand side with $|t-v| \ge U$ is 
$\ll U^{-A} X^{3+k/3}$
for any $A \ge 1$. Thus, taking $A=(4+k/3)/\eta$ the portion
of the above integral with $|t-v| \ge U$ is $\ll X^{-1}$. To bound the remaining portion of the integral apply the functional equation
to see that it is
\begin{equation} \notag
\begin{split} 
&\ll X^{1-2\delta} \iint_{|v-t| \le U} |\zeta^k(\tfrac12+\delta-it)\zeta^k(\tfrac12+\delta+iv)|
(tv)^{\delta k-1} dv \, dt\\
&\ll U X^{1-2\delta} \int_{Y-U}^{X+U} |\zeta(\tfrac12+\delta+it)|^{2k}\frac{dt}{t^{2-2\delta k}} \ll X^{1-2\delta+\eta} Y^{2\delta k-1+\epsilon},
\end{split}
\end{equation}
where in the last step we have used \eqref{int asump}.
\end{proof}

\begin{proof}[Proof of Propositions \ref{3prop} and \ref{main prop}]

To prove Proposition \ref{3prop} first apply
Lemmas \ref{contour} and \ref{plancherel}, where the smooth function $w$ is taken so that it majorizes the indicator function of the interval $[1,2]$. A result of Heath-Brown ~\cite{HB2} allows us to 
take $\delta=\frac{1}{12}$. Also, Weyl's bound gives $\lambda=\frac16$. 
It follows for $0<\theta\le \tfrac12$ that
\begin{equation} \label{d3 bound}
\frac1X
\int_X^{2X} \bigg( \Delta_3(x)-
\frac{x^{\frac{1}{3}}}{\pi \sqrt{3}}
\sum_{n \le \frac1x \cdot X^{1+\theta}}
\frac{d_3(n)}{n^{\frac23}}\cos\big(6\pi \sqrt[3]{nx}\big) \bigg)^2 \,
dx \ll X^{\frac23-\frac{\theta}{6}+\varepsilon}.
\end{equation}
If we assume the Lindel\"of hypothesis we may take $\delta=\lambda=0$. Arguing in the same way as before
we get
for $k \ge 3$ and $0< \theta \le \frac{1}{k-1}$ that
\begin{equation} \label{dk bound}
\frac1X
\int_X^{2X} \bigg( \Delta_k(x)-
\frac{x^{\frac{1}{2}-\frac{1}{2k}}}{\pi \sqrt{k}}
\sum_{n \le \frac1x \cdot X^{1+\theta}}
\frac{d_k(n)}{n^{\frac12+\frac{1}{2k}}}\cos\Big(2\pi k \sqrt[k]{nx}+\frac{k-3}{4}\pi\Big) \bigg)^2 \,
dx \ll X^{1-\frac{(1+\theta)}{k}+\varepsilon}.
\end{equation}

To complete the proof, we will now remove the dependence on the variable $x$ from the length of the sum.
Let $a_n=d_k(n) n^{-\frac12-\frac{1}{2k}}e\Big(\frac{k-3}{8}\Big)$, where $e(x)=e^{2\pi ix}$, and
integrate term-by-term to see
\[
 \int_X^{2X} \bigg|\sum_{\frac1x \cdot X^{1+\theta} \le n \le  X^{\theta}}
a_n e(k \sqrt[k]{nx})
\bigg|^2 \, dx
=\sum_{\frac12 \cdot X^{\theta} \le m,n \le  X^{\theta}} a_n\overline{a}_m \int_{Z}^{2X}
 e(k\sqrt[k]{x}(\sqrt[k]{n}-\sqrt[k]{m})) \, dx,
\]
where  $Z=\max(m^{-1},n^{-1}) \cdot X^{1+\theta}$.
The diagonal terms with $m=n$ are $\ll X^{1-\frac{\theta}{k}+\varepsilon}$. To bound the off-diagonal terms
with $m \neq n$ 
we integrate by parts to see that the above integral is $\ll X^{1-\frac1k}/|\sqrt[k]{m}-\sqrt[k]{n}|$.
Also, for $m>n$ we use the bound $\sqrt[k]{m}-\sqrt[k]{n} \gg (m-n) m^{\frac1k-1}$. 
Hence, by symmetry, the off-diagonal terms are bounded by
\begin{equation} \notag
\begin{split}
\ll X^{1-\frac1k} \sum_{\substack{\frac12 \cdot X^{\theta}\le m,n \le X^{\theta} \\
m > n} } \frac{|a_m a_n|}{|\sqrt[k]{n}-\sqrt[k]{m}|} 
\ll& X^{1-\frac1k} \sum_{\substack{\frac12 \cdot X^{\theta}\le m,n \le X^{\theta} \\
m > n} } \frac{|a_m a_n| m^{1-\frac1k}}{|m-n|} \\
\ll& X^{1-\frac1k} X^{\theta \cdot \frac{(k-1)}{k}} \log X  \sum_{\frac12 \cdot X^{\theta}\le n \le X^{\theta}} |a_n|^2.
\end{split}
\end{equation}
Since $\theta \le \frac{1}{k-1}$ this is $\ll X^{1-\frac1k+\theta \cdot \frac{(k-1)}{k}-\frac{\theta}{k}+\varepsilon} 
\ll X^{1-\frac{\theta}{k}+\varepsilon}$. Thus, Proposition \ref{3prop} follows from this and \eqref{d3 bound}. Proposition \ref{main prop} follows from this and \eqref{dk bound}.
\end{proof}

\section{The proofs of Theorems \ref{3thm}, \ref{main thm} and \ref{notsoshort}}

\begin{lemma} \label{MVT}
Suppose $0 \le \alpha \le 1$ and write $e(x)=e^{2\pi i x}$.
For any complex numbers $a_n$ we have
\begin{equation} \notag
\begin{split}
& \int_{X}^{2X} x^{\alpha} \bigg|\sum_{1 \le n \le N} a_n e(k \sqrt[k]{nx}) \bigg|^2
 \, dx \\
& \qquad \qquad \qquad \qquad = \sum_{1 \le n \le N} |a_n|^2 \cdot
\Big(\frac{2^{1+\alpha}-1}{1+\alpha} X^{1+\alpha}+O\Big(X^{1+\alpha-\frac1k} N^{1-\frac1k} \log N \Big)\Big) .
\end{split}
\end{equation}
\end{lemma}

\begin{proof} 
Integrating term-by-term we see that the diagonal terms give the main term. To bound the off-diagonal terms
we argue as in the previous proof. Integrate by parts and then use the estimate $|\sqrt[k]{m}-\sqrt[k]{n}| \gg |m-n| (\max(m,n))^{\frac1k-1}$ to get that
\begin{equation} \notag
\begin{split}
 \sum_{ \substack{1 \le m,n \le N \\ m >n}} |a_m \overline{a_n}| \bigg| \int_{X}^{2X} x^{\alpha} e(k \sqrt[k]{nx}) 
 \, dx \bigg|
\ll& X^{1+\alpha-\frac1k} \sum_{\substack{1 \le m,n \le N \\
m > n} } \frac{|a_m a_n|}{|\sqrt[k]{n}-\sqrt[k]{m}|} \\
\ll& X^{1+\alpha-\frac1k} \sum_{ \substack{ 1 \le m,n \le N \\
m > n} } \frac{|a_m a_n| m^{1-\frac1k}}{|m-n|} \\
\ll& X^{1+\alpha-\frac1k} N^{1-\frac1k} \log N  \sum_{1 \le n \le N} |a_n|^2.
\end{split}
\end{equation}
\end{proof}

Write
\[
\mathcal M_k(N,L)= X^{1-\frac1k} \cdot \frac{2\,(2^{2-\frac1k}-1)}{ \pi^2 k(2-\frac{1}{k})}\sum_{n \le N} \frac{d_k^2(n)}{n^{1+\frac1k}}\sin^2\Big(\pi \frac{\sqrt[k]{n}}{L} \Big).
\]

\begin{lemma} \label{p est}
Let $\varepsilon_1>0$ and suppose that $0<\theta \le \frac{1}{k-1}-\varepsilon_1$ and $L \ge 2$.
Then there exists $\varepsilon_2>0$ such that 
\begin{equation} \notag
\begin{split}
\frac1X \int_X^{2X} 
\Big(P_k \Big(x+\frac{x^{1-\frac1k}}{L};\theta   \Big)-P_k(x;\theta)\Big)^2 \, dx
=&\mathcal M_k(X^{\theta},L)\Big(1+O(X^{-\varepsilon_2}) \Big)\\
&+O\Big( \frac{X^{1-\frac1k-\varepsilon_2}}{L}\Big).
\end{split}
\end{equation}
\end{lemma}
\begin{proof}
Write
\[
M(x)=\frac{-2 \cdot x^{\frac12-\frac{1}{2k}}}{\pi \sqrt{k}} \sum_{n \le  X^{\theta}} \frac{d_k(n)}{n^{\frac12+\frac{1}{2k}}}
\sin\Big(\pi  \frac{\sqrt[k]{n}}{L} \Big)
\sin\Big(2\pi k \sqrt[k]{n}\Big(\sqrt[k]{x}+\frac{1}{2kL}\Big) +\frac{k-3}{4} \pi \Big).
\] 
It follows from some basic manipulations that
\begin{equation} \label{simp manip}
P_k\Big( \Big(\sqrt[k]{x}+\frac{1}{kL}\Big)^k;\theta   \Big)-P_k(x;\theta)
=M(x)+\mathcal R(x)
\end{equation}
where, for $X<x\le 2X$,
\[
\mathcal R(x)=
O\bigg(\frac{x^{\frac12-\frac{3}{2k}}}{L}\bigg|
\sum_{n \le  X^{\theta}} \frac{d_k(n)}{n^{\frac12+\frac{1}{2k}}}
\cos\Big(2\pi k \sqrt[k]{n}\Big(\sqrt[k]{x}+\frac{1}{kL} \Big)+\frac{k-3}{4} \pi \Big)\bigg|\bigg).
\]
Now write $a_n=d_k(n)n^{-\frac12-\frac{1}{2k}} \sin\Big(\pi  \frac{\sqrt[k]{n}}{L} \Big) e\Big(\frac{\sqrt[k]{n}}{2L}+ \frac{k-3}{8 }\Big)$, so that
\[
M(x)=
\frac{-x^{\frac{1}{2}-\frac{1}{2k}}}{ \pi i \sqrt{k}} \Bigg(
\sum_{n \le X^{\theta}} a_n e(k \sqrt[k]{nx})- \overline{\sum_{n \le X^{\theta}} a_n e(k \sqrt[k]{nx})}\Bigg).
\] 
Using Lemma \ref{MVT} it is not hard to see that
\begin{equation} \label{sig bd}
\frac1X \int_{X}^{2X}
M(x)^2 \, dx=\mathcal M_k(X^{\theta}, L) \cdot \Big(1+O(X^{\theta(1-\frac1k)-\frac1k} \log X)\Big) .
\end{equation}
Also, by Lemma \ref{MVT} it follows that 
\begin{equation} \label{r bd}
\frac1X \int_{ X}^{2X} 
\mathcal R(x)^2  \, dx \ll \frac{X^{1-\frac{3}{k}}}{L^2}
\sum_{n \le X^{\theta}} \frac{d_k^2(n)}{n^{1+\frac{1}{k}}} \ll \frac{X^{1-\frac{3}{k}}}{L^2}.
\end{equation}

Next, to shorten notation write
 \[
\Sigma(x)= \sum_{n \le X^{\theta}} \frac{d_k(n)}{n^{\frac12+\frac{1}{2k}}} \cos\Big(2\pi k x \sqrt[k]{n}+ \frac{k-3}{4} \cdot \pi \Big).
\]
 Also, let $x_1=\sqrt[k]{x}+\frac{1}{kL}$ and
$x_2=\Big( x+\frac{x^{1-\frac1k}}{L}\Big)^{\frac{1}{k}}$.
We have for $X \le x \le 2X$ that
\begin{equation} \notag
\begin{split}
|x_1^{\frac{k-1}{2}} \Sigma(x_1)-x_2^{\frac{k-1}{2}} \Sigma(x_2)| \ll & x_2^{\frac{k-1}{2}} |\Sigma(x_1)-\Sigma(x_2)|+|x_2^{\frac{k-1}{2}}-x_1^{\frac{k-1}{2}}| |\Sigma(x_1)| \\
\ll& |x_2-x_1| \cdot \Big(X^{\frac12-\frac{1}{2k}}\sum_{n \le X^{\theta}} \frac{d_k(n)}{n^{\frac12-\frac{1}{2k}}}+ X^{\frac{k-3}{2k}} |\Sigma(x_1)| \Big) \\
\ll & \frac{1}{L^2 X^{\frac1k}} \cdot \Big( X^{\frac12-\frac{1}{2k}+\theta(\frac12+\frac{1}{2k})+\frac{\varepsilon_1}{2k}}+ X^{\frac{k-3}{2k}} |\Sigma(x_1)| \Big).
\end{split}
\end{equation}
Thus, applying Lemma \ref{MVT} and using that $\theta \le \frac{1}{k-1}-\varepsilon_1$ we have
\[
\begin{split}
\frac1X \int_{X}^{2X} \Big(P_k\Big( \Big(\sqrt[k]{x}+\frac{1}{kL}\Big)^k; \theta \Big)-P_k\Big( x+\frac{x^{1-\frac1k}}{L}; \theta\Big)\Big)^2 \, dx 
 \ll&  \frac{X^{1-\frac{3}{k}+\theta(1+\frac{1}{k})+\frac{\varepsilon_1}{k}}}{L^4}+\frac{X^{1-\frac5k} }{L^4} \\
\ll& \frac{X^{1-\frac{1}{k}
+\frac{3-k}{k(k-1)}-\varepsilon_1}}{L^4}+
\frac{X^{1-\frac5k}}{L^4}.
\end{split}
\]
Combining this with \eqref{simp manip}, \eqref{sig bd}, and \eqref{r bd}
and then applying Cauchy-Schwarz we complete the proof.
\end{proof}

\begin{lemma} \label{partial summation}
Let $\theta>0$ and $k \ge 2$.
Suppose that $2 \le L=o(X^{\theta/k})$ as $X \rightarrow \infty$.
We have
\[
\sum_{n \le X^{\theta}} \frac{d_k^2(n)}{n^{1+\frac1k}}\sin^2\Big(\pi  \frac{\sqrt[k]{n}}{L} \Big)
=\frac{k^{k^2} \pi^2}{2 \, \Gamma(k^2)} \, a_k \cdot \frac{(\log L)^{k^2-1}}{L}+O\bigg(\frac{(\log L)^{k^2-2}}{L}+\frac{(\log L)^{k^2-1}}{X^{\theta/k}} \bigg),
\]
where $a_k$ is as given in \eqref{constants}.
\end{lemma}
\begin{proof}
We first require an estimate for the summatory function of $d_k(n)^2$, which follows from a standard argument that we will briefly sketch. Start with the generating series
\[
G(s) =\sum_{n =1}^{\infty} \frac{d_k(n)^2}{n^s}
=\prod_p\Big(\sum_{j=0}^{\infty} \frac{d_k(p^j)^2}{p^{js}} \Big)=\prod_p
\Big(\sum_{j=0}^{\infty}\Big(\frac{\Gamma(k+j)}{\Gamma(k)j!}\Big)^2\frac{1}{p^{js}} \Big)=
\zeta^{k^2} (s) g(s).
\]
Here the function $g(s)$ is analytic for $\tmop{Re}(s)>\tfrac12$
and is given by
\[
g(s)=
 \prod_p
\Big( (1-p^{-s})^{k^2} \sum_{j=0}^{\infty}\Big(\frac{\Gamma(k+j)}{\Gamma(k)j!}\Big)^2\frac{1}{p^{js}} \Big).
\]
Also, note that $g(s)$ is bounded for $\tmop{Re}(s)>\frac12+\varepsilon$ (see pages 173-174 of Titchmarsh ~\cite{Titchmarsh}).
 Applying Perron's formula, shifting contours of integration, and using Theorem 7.7 of Titchmarsh \cite{Titchmarsh} one can show that
\begin{equation} \label{summatory}
\sum_{n \le N} d_k(n)^2=N \, Q_{k^2-1}(\log N)+O(N^{1-\frac{1}{k^2}+\varepsilon}),
\end{equation}
where $Q_{k^2-1}(x)=\sum_{j=0}^{k^2-1} b_j x^j$
and
$
b_{k^2-1}=a_k/\Gamma(k^2)
$, where $a_k$ is the arithmetic factor in \eqref{constants}.

Using \eqref{summatory} we have that
\begin{equation} \notag
\begin{split}
\sum_{n \le  X^{\theta}} \frac{d_k^2(n)}{n^{1+\frac1k}}\sin^2\Big(\pi  \frac{\sqrt[k]{n}}{L} \Big)  
=\int_1^{X^{\theta}}
\frac{Q_{k^2-1}(\log x)+Q_{k^2-1}'(\log x)}{x^{1+\frac1k}} \sin^2\Big(\pi  \frac{\sqrt[k]{x}}{L} \Big) \, dx
+O\Big(\frac{1}{L} \Big).
\end{split}
\end{equation}
Make the change of variables $u= \sqrt[k]{x}/L$ and assume that $X^{\theta/k}/L
\rightarrow \infty$. The integral on the right-hand side equals
\begin{equation} \notag
\begin{split}
&\frac{k}{L} \int_{1/L}^{X^{\theta/k}/L} \Big( Q_{k^2-1}(k \log (Lu)))+Q_{k^2-1}'(k \log (Lu))\Big)
\frac{\sin^2(\pi  u )}{u^{2}}  \, du\\
& \qquad  =
\frac{b_{k^2-1}  \cdot k^{k^2} (\log L)^{k^2-1}}{L } \int_{1/L}^{X^{\theta/k}/L} \frac{\sin^2(\pi  u )}{u^{2}}  \, du
+O\bigg(\frac{(\log L)^{k^2-2}}{L}\bigg) \\
& \qquad  =
\frac{b_{k^2-1} \cdot k^{k^2} (\log L)^{k^2-1}}{L } \int_{0}^{\infty} \frac{\sin^2(\pi  u )}{u^{2}}  \, du
+O\bigg(\frac{(\log L)^{k^2-2}}{L}+\frac{(\log L)^{k^2-1}}{X^{\theta/k}} \bigg).
\end{split}
\end{equation}
Note that $\int_0^{\infty} \frac{\sin^2 (\pi u)}{u^2} \, du=\frac{ \pi^2}{2}.$ 
\end{proof}

\begin{proof}[Proof of Theorems \ref{3thm} and \ref{main thm}]
Let 
\[S_k(x)=\Delta_k\Big( x+\frac{x^{1-\frac1k}}{L}\Big)-\Delta_k(x) \quad
\mbox{ and }  \quad
\mathcal P_k(x)=P_k\Big( x+\frac{x^{1-\frac1k}}{L};\theta   \Big)-P_k(x;\theta).
\]
From Lemmas \ref{p est} and \ref{partial summation} we deduce that for $0<\theta \le \frac{1}{k-1}-\varepsilon$ 
and $L =o(X^{\theta/k})$ that
\begin{equation} \label{mean sq}
\frac{1}{X} \int_{X}^{2X} \mathcal P_k(x;\theta)^2 \, dx
=C_k \frac{X^{1-\frac1k}}{L} (\log L)^{k^2-1}\Big(1+O\Big(\frac{1}{\log L} \Big)\Big)+O\Big(X^{1-\frac{(1+\theta)}{k}+\varepsilon} \Big),
\end{equation}
where $C_k$ is as given in \eqref{constants}.
Proposition \ref{3prop} states for $0<\theta \le \tfrac12-\varepsilon$ that
\[
\frac{1}{X} \int_{X}^{2X}
\Big(S_3(x)-\mathcal P_3(x;\theta) \Big)^2  \, dx
\ll X^{\frac23-\frac{\theta}{6}+\varepsilon}.
\]
If $\theta =  \frac12-\varepsilon$ and $L \ll X^{\frac{1}{12}-2\varepsilon}$ this is smaller than \eqref{mean sq}, with $k=3$,
with a power savings in $X$.
Now apply Cauchy-Schwarz to see for $L \ll X^{\frac{1}{12}-2\varepsilon}$ that
\[
\frac{1}{X} \int_{X}^{2X} S_3(x)^2 \, dx
=C_3 \frac{X^{2/3}}{L} (\log L)^{k^2-1}+O\bigg(\frac{X^{2/3}}{L} (\log L)^{k^2-2}\bigg).
\]
This proves Theorem \ref{3thm}. Theorem \ref{main thm} follows from a similar argument, except that here we use Proposition \ref{main prop} with $\theta=\frac{1}{k-1}-\varepsilon$ in place of Proposition \ref{3prop}, so that we may take $2 \le L \ll X^{\frac{1}{k(k-1)}-2\varepsilon}$.
\end{proof}

\begin{proof}[Proof of Theorem \ref{notsoshort}]
Write
$a_n=d_k(n) n^{-\frac12-\frac{1}{2k}}e\Big(\frac{k-3}{8}\Big)$
so that
\[
P_k(x;\theta)=\frac{x^{\frac{1}{2}-\frac{1}{2k}}}{2\pi \sqrt{k}} \bigg(
\sum_{n \le  X^{\theta}} a_n e(k\sqrt[k]{nx})+\overline{\sum_{n \le  X^{\theta}} a_n e(k\sqrt[k]{nx})}\bigg).
\]
Applying Lemma \ref{MVT} it is not difficult to
see
that for $\varepsilon \le \theta \le \frac{1}{k-1}-\varepsilon$ that as $X \rightarrow \infty$
\begin{equation} \label{2vars}
\begin{split}
\frac{1}{X}\int_{X}^{2X} \Big(P_k(x;\theta) \Big)^2  dx
\sim \frac{ B_k}{2}  \cdot X^{1-\frac1k}.
\end{split}
\end{equation}
Next note that $(X-H)^{1-\frac1k} \sim X^{1-\frac1k}$ for $H=o(X)$ as $X \rightarrow \infty$. Using this estimate, making
the change of variables $u=x+H$ and applying Lemma \ref{MVT} one has for $\varepsilon \le \theta \le \frac{1}{k-1}-\varepsilon$ that as $X \rightarrow \infty$
\begin{equation} \label{3vars}
\frac{1}{X}\int_{X}^{2X} \Big(P_k(x+H;\theta) \Big)^2  dx\sim \frac{ B_k}{2}  \cdot X^{1-\frac1k}.
\end{equation}

We next estimate the covariance term.
Let
\[
I=\frac{1}{X}
\int_{X}^{2X} 
\frac{x^{1-\frac{1}{k}}}{4\pi^2 k}
\sum_{m,n \le  X^{\theta}} a_m \overline{a_n} e(k(\sqrt[k]{m(x+H)} - \sqrt[k]{nx}) \, dx
\]
and
\[
J=\frac{1}{X}
\int_{X}^{2X} 
\frac{x^{1-\frac{1}{k}}}{4\pi^2 k}
\sum_{m,n \le  X^{\theta}} a_m a_n e(k(\sqrt[k]{m(x+H)} + \sqrt[k]{nx}) \, dx.
\]
It follows that
\begin{equation}\label{covarsetup}
\frac1X
\int_{X}^{2X} P_k(x+H;\theta) P_k(x;\theta)  \, dx=2 \tmop{Re}\big(I+J\big).
\end{equation}
We will now bound $I$ and assume that $H=o(X^{1-\theta})$. Note that for real numbers $x,y>0$ we have $|\sqrt[k]{x}-\sqrt[k]{y}| \gg |x-y| (\max(x,y))^{\frac1k-1}$.  So for $m,n \le X^{\theta}$ with $m \ne n$ we have uniformly for $X \le x \le 2X$ that
\[
\bigg|x^{\frac1k-1}\Big(\sqrt[k]{m}\Big(1+\frac{H}{x}\Big)^{\frac1k-1}-\sqrt[k]{n}\Big) \bigg| \gg X^{\frac1k-1} |m-n+o(1)| (\max(m,n))^{\frac1k-1}>0.
\]
Using this bound along with Lemma 4.3 of Titchmarsh \cite{Titchmarsh}, or alternatively integrating by parts, we have
\[
\frac{1}{X} \bigg|
\int_{X}^{2X} x^{1-\frac1k} \, e(k(\sqrt[k]{m(x+H)} - \sqrt[k]{nx}) \, dx \bigg| \ll
\begin{cases}
 \displaystyle \frac{X^{1-\frac2k} (\max(m,n))^{1-\frac1k}}{|m-n|} \quad \mbox{    if } m \ne n, \\
\displaystyle \frac{X^{2-\frac2k}}{H \sqrt[k]{n}} \qquad \mbox{   if } m = n.
\end{cases}
\]
Thus, for $\theta< \frac{1}{k-1}-3\varepsilon$ the contribution of the terms with $m \ne n$ to $I$
is 
\begin{equation} \notag
\begin{split}
\ll X^{1-\frac{2}{k}} \sum_{\substack{m,n \le X^{\theta} \\ m \ne n}} \frac{|a_m a_n| (\max(m,n))^{1-\frac1k}}{|m - n|} 
\ll X^{1-\frac{2}{k}+\theta(1-\frac1k)+\varepsilon} \sum_{n \le X^{\theta}} |a_n|^2 
 \ll  X^{1-\frac1k-\varepsilon}.
\end{split}
\end{equation}
The terms with $m=n$ contribute
\[
\ll \frac{X^{2-\frac2k}}{H} \sum_{n \le X^{\theta}} \frac{|a_n|^2}{\sqrt[k]{n}}
\ll  \frac{X^{2-\frac2k}}{H}.
\]
It follows for $\theta<\frac{1}{k}-\varepsilon$ and $X^{1-\frac1k+\varepsilon} \ll H=o(X^{1-\theta})$ that $I =O(X^{1-\frac1k-\varepsilon})$. The proof of the analogous bound for $J$ follows from a similar, but easier argument that we will omit.
Using these bounds in \eqref{covarsetup} we get for $\theta<\frac{1}{k}-\varepsilon$ and $X^{1-\frac1k+\varepsilon} \ll H=o(X^{1-\theta})$ that
\[
\frac1X
\int_{X}^{2X}P_k(x+H;\theta)  P_k(x;\theta) \, dx =O(X^{1-\frac1k-\varepsilon}).
\]
Therefore, combining this with \eqref{2vars} and \eqref{3vars} we obtain for $\theta<\frac{1}{k}-\varepsilon$ and $X^{1-\frac1k+\varepsilon} \ll H=o(X^{1-\theta})$ that as $X \rightarrow \infty$
\begin{equation} \label{polydone}
\frac{1}{X}\int_{X}^{2X} \Big(P_k\Big(x+H; \theta\Big)-P_k(x; \theta)\Big)^2 \, dx \sim B_k  \cdot X^{1-\frac1k}.
\end{equation}

To complete the proof for $k=3$ we use \eqref{polydone}, Proposition \ref{3prop} with $\theta=12\varepsilon$, and then Cauchy-Schwarz.
For $k \ge 4$  one argues in the same way only now use Proposition \ref{main prop} with $\theta=2 k \varepsilon$ in place of Proposition \ref{3prop}. For $k=2$ we use a classical estimate of Titchmarsh (see \cite{Titchmarsh} (12.4.4)) which unconditionally implies \eqref{varbdbest} in the case $k=2$, so that this case now follows as well.
\end{proof}

\noindent
\textbf{Acknowledgments.}  I would like to thank Zeev Rudnick for his encouragement and for many helpful and interesting discussions.

\end{document}